\documentclass[12pt]{amsart}
\usepackage{amssymb,latexsym,amsthm,mathptmx}
\usepackage{amsmath,amsfonts,amssymb}
\usepackage{graphicx}
\usepackage{color}
\usepackage{mathrsfs}

\newcommand{\ch}{\operatorname{Ch}}

\newcommand{\bbT}{\mathbb{T}}

\newcommand{\mxa}{M_{x,\alpha}}
\newcommand{\fxa}{F_{x,\frac{\alpha}{|\alpha|}}}
\newcommand{\fxma}{F_{x,\frac{-\alpha}{|\alpha|}}}
\newcommand{\Tfxa}{F_{\phi(x),\frac{\alpha}{|\alpha|}\tau(x,1)}}
\newcommand{\Tfxba}{F_{\phi(x),\frac{\overline{\alpha}}{|\alpha|}\tau(x,1)}}
\newcommand{\Tfxma}{F_{\phi(x),\frac{-\alpha}{|\alpha|}\tau(x,1)}}
\newcommand{\Tfxmba}{F_{\phi(x),\frac{-\overline{\alpha}}{|\alpha|}\tau(x,1)}}

\newtheorem{theorem}{Theorem}[section]
\newtheorem{lemma}[theorem]{Lemma}

\theoremstyle{definition}
\newtheorem{definition}[theorem]{Definition}

\theoremstyle{remark}

\begin{document}
\author{
Osamu~Hatori
}
\address{
Institute of Science and Technology,
Niigata University, Niigata 950-2181, Japan
}
\email{hatori@math.sc.niigata-u.ac.jp
}

\author{
Shiho Oi
}
\address{
Department of Mathematics, Faculty of Science, 
Niigata University, Niigata 950-2181, Japan.
}
\email{shiho-oi@math.sc.niigata-u.ac.jp
}

\author{
Rumi Shindo Togashi}
\address{
National Institute of Technolory, Nagaoka College
888 Nishikatakai, Nagaoka, Niigata 940-8532, Japan
}
\email{rumi@nagaoka-ct.ac.jp
}


\title[]
{Tingley's problems on uniform algebras}




\begin{abstract}
We prove that a surjective isometry between the unit spheres of two uniform algebras is extended to a surjective real-linear isometry between the uniform algebras. It provides the first positive solution to Tingley's problem on a Banach space,  without being a Hilbert space,  consisting of analytic functions.
\end{abstract}
\maketitle
\section{Introduction}\label{sec1}
Tingley's problem asks the extensibility of a surjective isometry between the unit spheres of two Banach spaces. 
In 1987, Tingley \cite{tingley} proposed the problem. Since then
a lot of work has been done trying to solve this. Despite its simple statement, Tingley's problem is a hard problem which remains unsolved for several types of Banach spaces.
No counterexample is known. Due to \cite[p.730]{yangzhao} Ding was the first to consider Tingley's problem between different type of spaces \cite{ding2003B}.
A Banach space $E$ satisfies the Mazur-Ulam property if  a surjective isometry  between the unit spheres of $E$ and any Banach space is extended to a surjective isometry between the whole spaces (cf. \cite{chengdong}).

Wang \cite{wang} seems to be the first to solve Tingley's problem between specific spaces. He dealt with $C_0(\Omega)$, the space of all complex-valued continuous functions on a locally compact Hausdorff space $\Omega$ which vanish at infinity. 
Although we do not mention each of the literatures, a considerable number of interesting  results have shown that Tingley's problem has a positive solution on $\ell^p$ spaces, $L^p$ spaces, $C_0(\Omega)$ spaces, certain finite dimensional Banach spaces, and so on. 

Results on the classical Banach spaces satisfying the Mazur-Ulam property are a little bit reduced. Including the space of all real null sequences, the space of all bounded real-valued functions on a discrete set, the space of all real-valued continuous functions on a compact Hausdorff space, the space of all  $p$-integrable or  essentially bounded real-valued functions on a $\sigma$-finite measure space for $1\le p<\infty$ are studied.
In \cite{thl2013} 
Tan, Huang and Liu introduced the notion of generalized lush spaces and local GL spaces and proved that every local GL space satisfies the Mazur-Ulam property.
The main result in \cite{jmpr2019} probably provides the first example of a complex Banach space which is not a Hilbert space  satisfying the Mazur-Ulam property (cf. \cite{peralta2019a}).

Besides the large list of classical Banach spaces giving positive solutions,
a series of papers by Tanaka \cite{tanaka,tanaka2017,tanaka2017b} on the algebra of complex matrices, a finite dimensional $C^*$-algebra and a finite von Neumann algebra opened up another direction of study on Tingley's problem. 
After 
	\cite{cp2019,fgpv2017,fjp2020,fp2018a,fp2018b,fp2017,lnw2020,mori,pt2019}
new achievements by Mori and Ozawa \cite{moriozawa} prove that the Mazur-Ulam  property is satisfied by unital $C^*$-algebras and real von Neumann algebras.
The result proving that all general JBW*-triples satisfy the Mazur-Ulam Property is established by 
Becerra-Guerrero, Cueto-Avellaneda, Fern\'andez-Polo and Peralta \cite{bcfp}  and 
Kalenda and Peralta \cite{kp}. 
Fern\'andez-Polo and Peralta \cite{fp2018c}  prove that Tingley's problem admits a positive solution for surjective isometries between the unit spheres of two 
weakly compact JB*-triples. The Mazur-Ulam property for all weakly compact JB*-triples is established by a recent paper by Peralta \cite{peraltaJB*}.

Nagy proved  a variant of Tingley's problem in \cite[Theorem]{nagy2018} and conjectured that the similar conclusion holds for every complex Hilbert space. Peralta \cite[Theorem 6.10]{peralta2018} completed a proof of Nagy's conjecture and proposed a more general problem (cf. \cite{peralta2019b}). Combining with a problem posed by
Mori  for noncommutative $L^p$-spaces \cite[Problem 6.3]{mori} Leung, Ng and Wong \cite[Problem 1]{lnw2021} extended a problem of Peralta for ordered Banach spaces (see also \cite{lnw2018}). They provide results on a surjective isometry between the positive unit sphere of a function space $L^p(\mu)$ or $C(X)$ \cite{lnw2021}.

The reader is referred to the surveys \cite{ding2009,peralta2018,yangzhao} and Introduction  of \cite{cueto} for details about Tingley's problem.

In this paper we add a positive solution to Tingley's problem. We prove that a surjective isometry between the spheres of two uniform algebras is extended to a surjective real-linear isometry between the whole uniform algebras. 
Let $X$ be a compact Hausdorff space and $C(X)$ the Banach algebra of all complex-valued continuous functions on $X$. 
A closed subalgebra of $C(X)$ which contains constants and separates the points of $X$ is called a uniform algebra on $X$. 
 A uniform algebra is called a function algebra in \cite{br}, which is the name of other object in some literatures recently. To avoid a confusion, we use the terminology ``a uniform algebra''. 
 Typical examples of  a uniform algebra  consist of analytic functions of one and several complex-variables such as the disk algebra, the polydisk algebra and the ball algebra. The algebra of all bounded analytic functions $H^\infty(D)$ on a certain domain $D$ is considered as a uniform algebra on the maximal ideal space. In fact, Theorem \ref{main} provides the first positive solution to Tingley's problem on a Banach space of analytic functions.
 
 For further information about uniform algebras, see \cite{br}. 
\section{Surjective isometries between the unit spheres of uniform algebras}\label{sec2}
 Throughout the paper, $A$ and $B$ are uniform algebras on compact Hausdorff spaces $X$ and $Y$, respectively.  The norm of a uniform algebra is denoted by $\|\cdot\|$. 
We denote the unit sphere of $A$ (resp. $B$) by $S(A)$ (resp. $S(B)$). The map $T:S(A)\to S(B)$ is always a surjective isometry in this paper.
We do not apply the notation of the Gelfand transform. The Gelfand transform of $f\in A $ is also denoted by $f$. We may assume $f\in A$ is also defined on the maximal ideal space $M_A$ of $A$ without confusion.

Our main result is the following.
\begin{theorem}\label{main}
Let $A$ and $B$ be uniform algebras on compact Hausdorff spaces $X$ and $Y$, respectively.
Suppose that $T:S(A)\to S(B)$ is a surjective isometry. Then $T$ is extended to a real-linear surjective isometry from $A$ onto $B$. In particular, $|T(1)|=1$ on $M_B$, and there exists a homeomorphism $\Psi:M_B\to M_A$ and possibly empty disjoint closed and open subsets $M_{B+}$ and $M_{B-}$ of $M_B$ with $M_{B+}\cup M_{B-}=M_B$ such that a surjective real-linear isometry $\widetilde{T}:A\to B$ defined by 
\begin{equation*}
\widetilde{T}(f)=T(1)\times 
\begin{cases}
\text{$f\circ\Psi$ \quad on $M_{B+}$} \\
\text{$\overline{f\circ \Psi}$ \quad  on $M_{B-}$}
\end{cases}
\end{equation*}
for every $f\in A$ extends $T$; i.e. $\widetilde{T}|S(A)=T$.
\end{theorem}
In the following sections we prepare several lemmata to prove Theorem \ref{main}. We exhibit a proof of the theorem in  Section \ref{sec7}.
\section{Maximal convex sets}\label{sec3}
The following lemma is well known (cf. \cite[p. 108]{mp} and \cite[p. 339]{tl2013}).  A proof of the existence of $\varphi$ in the unit sphere of the dual space $E^*$ appears in \cite[Lemma 3.3]{tanaka}. For a convenience of the readers we present a proof that $\varphi$ can be chosen so that it is an extreme point of the closed unit ball of $E^*$.
\begin{lemma}\label{mcsbybaos}
Let $E$ be a Banach space. Suppose that $F$ is a maximal convex subset of the unit sphere $S(E)$ of $E$.
Then there exists an extreme point $\varphi$ of the closed unit ball $B(E^*)$ of $E^*$ such that $\varphi^{-1}(1)\cap S(E)=F$.
\end{lemma}
\begin{proof}
It is well known that there exists $\chi$ in the unit sphere $S(E^*)$ of the dual space $E^*$ of $E$ such that $\chi^{-1}(1)\cap S(E)=F$ (cf. \cite[Lemma 3.3]{tanaka} ). We show that $\chi$ can be chosen so that it is an extreme point of the unit ball $B(E^*)$ of $E^*$. Put $\Phi=\{\chi\in S(E^*):\chi^{-1}(1)\cap S(E)=F\}$. It is routine work to show that $\Phi$ is a non-empty closed convex subset of  $B(E^*)$. Then  the Krein-Milman theorem asserts that there is an extreme point $\varphi$ in $\Phi$. We prove that $\varphi$ is also an extreme point in $B(E^*)$. Suppose that $\varphi=(\varphi_1+\varphi_2)/2$ for $\varphi_1,\varphi_2\in B(E^*)$. Then 
\[
1=\varphi(f)=(\varphi_1(f)+\varphi_2(f))/2
\]
and 
\[
|\varphi_1(f)|\le1,\,\,|\varphi_2(f)|\le 1
\]
assert that $\varphi_1(f)=\varphi_2(f)=1$ for every $f\in F$. Hence we have $\varphi_1$, $\varphi_2 \in \Phi$. As $\varphi$ is an extreme point of $\Phi$, we see that $\varphi_1=\varphi_2$, so that $\varphi$ is an extreme point of $B(E^*)$ and $\varphi^{-1}(1)\cap S(E)=F$.
\end{proof} 
We show that
a maximal convex subset of $S(A)$ for a uniform algebra $A$ has a specific form.
We say that a non-empty closed subset $K$ of $X$ is a peak set for $A$ if $K=f^{-1}(1)$ for a function $f\in A$ such that $\|f\|=1$ and $|f|^{-1}(1)=f^{-1}(1)$. Such a function $f$ is called a peaking function for $K$. 
An intersection of peak sets is called a weak peak set (peak set in the weak sense in \cite{br}). 
When a peak set (resp. weak peak set) is a singleton, the unique element in $K$ is called a peak point (resp. weak peak point). 
The Choquet boundary of $A$  is denoted by $\ch(A)$.
It is known that $\ch(A)$ consists of the weak peak points and $\ch(A)$ is a boundary for $A$ (cf. \cite[Theorems 2.2.9, 2.3.4]{br}).

For a uniform algebra $A$ every maximal convex subset of $S(A)$ is represented as  $\{f\in S(A):f(x)=\lambda\}$ for $x\in \ch(A)$ and  $\lambda\in \bbT$, where $\bbT$ denotes the unit circle in the complex plane throughout the paper.
\begin{lemma}\label{0}
Let $A$ be a uniform algebra on $X$. 
Then a subset $F$ of $S(A)$ is a maximal convex subset of $S(A)$ if and only if there exists a Choquet boundary point $x\in \ch(A)$ and $\lambda\in \bbT$ such that
$F=\{f\in S(A):f(x)=\lambda\}$.
\end{lemma}
\begin{proof}
Suppose that $F$ is a maximal convex subset of $S(A)$. By Lemma \ref{mcsbybaos} there exists an extreme point $\varphi$ in $B(A^*)$ with 
$F=\varphi^{-1}(1)\cap S(A)$. It is well known that each extreme point in $B(A^*)$ is of the form $\gamma\delta_x$ for $x\in \ch(A)$ and $\gamma\in \bbT$, where $\delta_x$ denotes the point evaluation at $x$ (cf. \cite[Corollary 2.3.6 and Definition 2.3.7]{fj1}).  
It follows that $F=\{f\in S(A):f(x)=\lambda\}$, where $\lambda=\overline{\gamma}$.

Conversely suppose that $x\in \ch(A)$, $\lambda\in \bbT$ and $F=\{f\in S(A):f(x)=\lambda\}$. By a simple calculation we have that $F$ is a convex subset of $S(A)$. We prove that it is maximal. Applying Zorn's lemma there exists a maximal convex subset $C$ of $S(A)$ with $F\subset C$. Then by the first part of the proof, there exist $x_0\in \ch(A)$ and $\lambda_0\in \bbT$ such that $C=\{f\in S(A):f(x_0)=\lambda_0\}$.  Thus we have
\[
\{f\in S(A):f(x)=\lambda\}\subset \{f\in S(A):f(x_0)=\lambda_0\}.
\]
We prove $x_0=x$. Suppose not.
Since a Choquet boundary point is a weak peak point \cite[Theorem 2.3.4]{br}, there exists $g\in A$ such that $g(x)=1=\|g\|$ and 
$|g(x_0)|<1/2$. Then $\lambda g\in \{f\in S(A):f(x)=\lambda\}$ and $\lambda g\not\in \{f\in S(A):f(x_0)=\lambda_0\}$, which is a contradiction. 
Thus  we have $x_0=x$, and  $\lambda_0=\lambda$ follows. Hence we have $C=F$.
\end{proof}
In the following we denote the maximal convex subsets of $S(A)$  by \[
F_{x,\lambda}
=\{f\in S(A):\,\,f(x)=\lambda\},
\]
for $x\in \ch(A)$ and $\lambda\in \bbT$. 
\section{Maps between maximal convex sets}\label{sec4}
The following may be known for the readers who are familiar with  uniform algebras. For the convenience of general readers we exhibit it with a short proof.  
\begin{lemma}\label{preparatory}
Suppose that $y_1,y_2\in \ch(A)$ are different points. Then for any pair $\mu_1,\mu_2\in \bbT$, there exists $g\in S(A)$ such that $g(y_1)=\mu_1$ and $g(y_2)=\mu_2$.
\end{lemma}
\begin{proof}
First we prove that $\{y_1,y_2\}$ is a weak peak set. Since $y_1$ and $y_2$ are weak peak points, there exist possibly infinitely many peak sets $K_t^1$ and $K_s^2$ such that $\{y_1\}=\cap_tK_t^1$ and $\{y_2\}=\cap_sK_s^2$. Let $f_t^1$ and $f_s^2$ be peaking functions for $K_t^1$ and $K_s^2$, respectively. Let $(1-f_t^1)^{\frac12}$ and $(1-f_s^2)^{\frac12}$ be the  principal values. By a simple calculation we see that $\exp(- (1-f_t^1)^{\frac12}(1-f_s^2)^{\frac12})$ is a peaking function for $K_t^1\cup K_s^2$. As $\{y_1,y_2\}=\cap_{t,s}(K_t^1\cup K_s^2)$ we conclude that $\{y_1,y_2\}$ is a weak peak set.

Since $A$ separates the points of $X$ and contains constant functions, it is straightforward to see that there exists $h\in A$ with $h(y_1)=\mu_1$ and $h(y_2)=\mu_2$. If $\|h\|=1$, then $h$ is the desired function $g$. Suppose that $\|h\|>1$. 
Put $L_0=\{y\in X:|h(y)|\ge 1+1/2\}$. For a positive integer $n$, put $L_n=\{y\in X:1+1/2^{n+1}\le |h(y)|\le 1+1/2^n\}$. Then $\cup_{n=0}^\infty L_n=\{y\in X:|h(y)|>1\}$. 
Recall that  $\{y_1,y_2\}$ is a weak peak set, say $\{y_1,y_2\}=\cap_iK_i$ where $K_i'$s are possibly infinitely many peak sets, and $\{y_1,y_2\}\cap L_0=\emptyset$. 
So  $L_0\subset Y\setminus\{y_1,y_2\}=\cup_i(Y\setminus K_i)$, where $Y\setminus K_i$ is an open set.
As $L_0$ is compact, there exists a finite number of $K_{i_1},\cdots, K_{i_m}$  such that $L_0 \subset \cup_{j=1}^m(Y\setminus K_{i_j})$.
Letting $K=\cap_{j=1}^mK_{i_j}$ we have
$\{y_1,y_2\}\subset K$ and $K\cap L_0=\emptyset$. 
Note that $K$ is a peak set with a peaking function $f=\prod_{j=1}^mf_j$, where $f_j$ is the corresponding peaking function for $K_{i_j}$.
Put $u_0=f^M$  with a sufficiently large natural number $M$. We may assume that $u_0\in B$ is a peaking function for $K$ such that $|u_0|<1/\|h\|$ on $L_0$. 
In the same way for each positive integer $n$, there exists a peaking function $u_n\in A$ which peaks on some peak set which contains $\{y_1,y_2\}$ such that $|u_n|<1/(2^n+1)$ on $L_n$. Put $u=u_0\sum_{n=1}^\infty(u_n/2^n)$. Then we see that $hu$ is the desired function $g$ as follows. Since $u=1$ on $\{y_1,y_2\}$, we have $hu(y_1)=\mu_1$ and $hu(y_2)=\mu_2$. We prove that $\|hu\|=1$. Let $y\in L_0$. Then
\[
|hu(y)|\le |h(y)||u_0(y)|\le 1.
\]
Let $y\in L_n$. Then
\begin{multline*}
|hu(y)|\le |h(y)|\left(\left|\frac{u_n(y)}{2^n}\right|+\sum_{k\ne n}\frac{1}{2^k}\right)\\
\le \left(1+\frac{1}{2^n}\right)\left(\frac{1}{2^n(2^n+1)}+1-\frac{1}{2^n}\right)=1.
\end{multline*}
Thus $|hu|\le 1$ on $\cup_{n=0}^\infty L_n$. Let $y\in X\setminus \cup_{n=0}^\infty L_n$. As  $\cup_{n=0}^\infty L_n=\{y\in Y:|h(y)|>1\}$ we have 
$|h(y)|\le 1$, and so $|hu(y)|\le 1$. It follows that $\|hu\|=1$. Thus $hu$ is the desired function $g$.
\end{proof}
In the following lemma, we establish a parametric correspondence between the maximal convex sets of the unit spheres which is a key gradient of the proof of the main result.  
To prove the lemma we apply the fact that a surjective isometry between the spheres of two Banach spaces preserves the maximal convex subsets of the unit spheres. This is originally due to Cheng and Dong  \cite[Lemma 5.1]{chengdong}. On the other hand, in the proof of  Lemma 5.1 they seem to apply a non-trivial argument such as  $C\cap E_1$ is a maximal convex subset of the unit sphere of a subspace $E_1$ of a given Banach space $E$ provided that  $C$ is a maximal convex subset of the unit sphere of $E$ without applying a further explanation. In fact, Remark 2.3 in \cite{tl2013} gives a counterexample that $C\cap E_1$ needs not be maximal, in general.  The authors of this paper are not sure that the proof of Cheng and Dong is valid or not at this point.
Fortunately Lemma 5.1 in \cite{chengdong} was later proved by Tanaka \cite[Lemma 3.5]{tanaka2014b} (cf. \cite[Lemma 6.3]{tanaka}).  Although the proof of  \cite[Lemma 3.5]{tanaka2014b} was done for the real case, it is also effective for the complex case by considering complex spaces as real spaces.
\begin{lemma}\label{mapsbtmcs}
There exists a bijection $\phi:\ch(A)\to \ch(B)$ and a map $\tau:\ch(A)\times {\mathbb{T}}\to {\mathbb{T}}$ such that 
\[
T(F_{x,\lambda})=F_{\phi(x),\tau(x,\lambda)}
\]
for every maximal convex subset $F_{x,\lambda}$ of $S(A)$. 
\end{lemma}
\begin{proof}
Let $F_{x,\lambda}$ and $F_{x',\lambda'}$ with $x,x'\in \ch(A)$ and $\lambda, \lambda'\in {\mathbb T}$ be maximal convex subsets of $S(A)$. 
First we point out that $F_{x,\lambda}=F_{x',\lambda'}$  ensures that $x=x'$ and $\lambda=\lambda'$. It is straightforward that the equality $x=x'$ holds by Lemma \ref{preparatory}.  Then $\lambda=\lambda'$ follows.  
By \cite[Lemma 3.5]{tanaka2014b}
$T(F_{x,\lambda})$ is a 
maximal convex subset of $S(B)$. Hence by Lemma \ref{0} we obtain $\phi(x,\lambda)\in \ch(B)$ and $\tau(x,\lambda)\in \bbT$ such that 
\[
T(F_{x,\lambda})=F_{\phi(x,\lambda),\tau(x,\lambda)}.
\]
We prove that $\phi(x,\lambda)$ does not depend on the second term $\lambda$.  Suppose that there exists $x\in \ch(A)$ and $\lambda, \lambda'\in \bbT$ such that $\phi(x,\lambda)\ne \phi(x, \lambda')$.
Then by Lemma \ref{preparatory} there exists $g\in S(B)$ such that 
\[
g(\phi(x,\lambda))=\tau(x,\lambda), \quad g(\phi(x,\lambda'))=\tau(x,\lambda').
\]
This means that $g\in F_{\phi(x,\lambda),\tau(x,\lambda)}=T(F_{x,\lambda})$ and $g\in F_{\phi(x,\lambda'),\tau(x,\lambda')}=T(F_{x,\lambda'})$.  Then there exists $f_1\in F_{x,\lambda}$ and $f_2\in F_{x,\lambda'}$ such that $T(f_1)=T(f_2)=g$. As $T$ is an injection, $f_1=f_2$. However it is impossible since 
$F_{x,\lambda}\cap F_{x,\lambda'}=\emptyset$. We have proved that $\phi(x,\lambda)$ is independent of $\lambda$. Thus we write 
$\phi(x)$ instead of $\phi(x,\lambda)$ and the conclusion holds.

Applying the same argument for $T^{-1}$ we get $\phi':\ch(B)\to \ch(A)$ and $\tau':\ch(B) \times {\mathbb T}\to {\mathbb T}$ such that $T^{-1}(F_{y,\lambda})=F_{\phi'(y),\tau'(y,\lambda)}$. Then we have
\[
F_{x,\lambda}=T^{-1}(T(F_{x,\lambda}))=T^{-1}(F_{\phi(x),\tau(x,\lambda)})=F_{\phi'(\phi(x)),\tau'(\phi(x),\tau(x,\lambda))}
\]
and 
\[
F_{y,\lambda}=T(T^{-1}(F_{y,\lambda}))=
T(F_{\phi'(y),\tau'(y,\lambda)})=F_{\phi(\phi'(y)), \tau(\phi'(y), \tau'(y,\lambda))}.
\]
Therefore $x=\phi'\circ\phi(x)$ for every $x\in \ch(A)$ and $y=\phi\circ\phi'(y)$ for every $y\in \ch(B)$ hold. Hence $\phi$ is a bijection.
\end{proof}
\begin{lemma}\label{txl}
Let $\ch(A)_+=\{x\in\ch(A):\frac{\tau(x,i)}{\tau(x,1)}=i\}$ and $\ch(A)_-=\{x\in \ch(A):\frac{\tau(x,i)}{\tau(x,1)}=\overline{i}\}$. Then $\ch(A)_+\cup\ch(A)_-=\ch(A)$. 
For every $\lambda\in\bbT$ we have
\begin{equation*}
\tau(x,\lambda) =
\begin{cases}
\lambda \tau(x,1),\qquad x\in \ch(A)_+, \\
\overline{\lambda}\tau(x,1), \qquad x\in \ch(A)_-.
\end{cases}
\end{equation*}
\end{lemma}
\begin{proof}
Let $x\in \ch(A)$ and $\lambda\in {\mathbb T}$. Since $|\tau(\cdot,\cdot)|=1$ on $\ch(A)\times {\mathbb T}$ we have
\begin{multline}\label{eq3-1}
|1-\lambda|= \|1-\lambda\|=\|T(1)-T(\lambda)\|=\|\tau(\cdot,1)-\tau(\cdot, \lambda)\| \\
\ge |\tau(x,1)-\tau(x,\lambda)|=\left|1-\frac{\tau(x,\lambda)}{\tau(x,1)}\right|.
\end{multline}
Since $\ch(A)$ is a boundary, we have $\{\gamma\}=\cap_{x\in \ch(A)}F_{x,\gamma}$ for every unimodular constant function $\gamma \in A$. Then by Proposition 2.3 in \cite{mori} we have that
\begin{multline}\label{eq3-1.5}
\{T(-\mu)\}=T(\cap_{x\in\ch(A)}F_{x,-\mu})=\cap_{x\in\ch(A)}T(F_{x,-\mu})\\
=\cap_{x\in \ch(A)}T(-F_{x,\mu})
=\cap_{x\in \ch(A)}(-T(F_{x,\mu}))\\
=-\cap_{x\in \ch(A)}T(F_{x,\mu})
=-T(\cap_{x\in \ch(A)}F_{x,\mu})=\{-T(\mu)\},
\end{multline}
hence $T(-\mu)=-T(\mu)$ for every unimodular constant function $\mu$.
Applying \eqref{eq3-1.5} for $\mu=1$ we have 
\begin{multline}\label{eq3-2}
|-1-\lambda|=\|1+\lambda\|=\|-(-1)+\lambda\|=\|-T(-1)+T(\lambda)\| \\
=\|T(1)+T(\lambda)\| 
=\|\tau(\cdot,1)+\tau(\cdot,\lambda)\|
\ge|\tau(x,1)+\tau(x,\lambda)|=\left|-1-\frac{\tau(x,\lambda)}{\tau(x,1)}\right|
\end{multline}
Since $\tau(x,\lambda)/\tau(x,1)$ is a unimodular complex number, we have 
by the equation \eqref{eq3-1} that $\tau(x,\lambda)/\tau(x,1)$ is on the arc through $1$ of the unit circle with the end points $\lambda$ and $\overline{\lambda}$.  By the equation \eqref{eq3-2} we obtain that $\tau(x,\lambda)/\tau(x,1)$ is on the arc through $-1$ of the unit circle with the end points $\lambda$ and $\overline{\lambda}$. Therefore we have $\tau(x,\lambda)/\tau(x,1)\in \{\lambda,\overline{\lambda}\}$. Note that such a way of a proof is called 
Yoko-Kuwagata.
Hence $\tau(x,\lambda)=\lambda \tau(x,1)$ if $\lambda\in {\mathbb R}$. 
Besides, $\tau(x,i)/\tau(x,1)\in \{\pm i\}$ and thus $\ch(A)_+\cup\ch(A)_-=\ch(A)$.

We now prove $\tau(x,\lambda)=\lambda\tau(x,1)$ if $x\in \ch(A)_+$ and $\lambda\in \bbT$. We may suppose that $\operatorname{Im}(\lambda)\ne 0$. In the case $\operatorname{Im}(\lambda)>0$, we have
\begin{multline*}
|i-\lambda|=\|i-\lambda\|=\|T(i)-T(\lambda)\|=\|\tau(\cdot,i)-\tau(\cdot,\lambda)\| \\
\ge |\tau(x,i)-\tau(x,\lambda)|=\left|\frac{\tau(x,i)}{\tau(x,1)}-\frac{\tau(x,\lambda)}{\tau(x,1)}\right|=\left|i-\frac{\tau(x,\lambda)}{\tau(x,1)}\right|.
\end{multline*}
Since $\frac{\tau(x,\lambda)}{\tau(x,1)}=\lambda$ or $\overline{\lambda}$ and $\operatorname{Im}\lambda >0$, the inequality $|i-\lambda|\ge \left|i-\frac{\tau(x,\lambda)}{\tau(x,1)}\right|$ asserts that $\frac{\tau(x,\lambda)}{\tau(x,1)}=\lambda$. Next we consider the case of 
 $\operatorname{Im}\lambda<0$. Applying \eqref{eq3-1.5} for $\mu=i$, by a similar calculation  in \eqref{eq3-2} we get
 \begin{multline*}
 |-i-\lambda|=\|T(-i)-T(\lambda)\|=\|-T(i)-T(\lambda)\|=\|-\tau(\cdot,i)-\tau(\cdot,\lambda)\| \\
 \ge\left|-\frac{\tau(x,i)}{\tau(x,1)}-\frac{\tau(x,\lambda)}{\tau(x,1)}\right|=\left|-i-\frac{\tau(x,\lambda)}{\tau(x,1)}\right|.
\end{multline*}
 As $\operatorname{Im}\lambda <0$, we have $\frac{\tau(x,\lambda)}{\tau(x,1)}=\lambda$.
 
Similarly, it can be proved that $\tau(x,\lambda)=\overline{\lambda}\tau(x,1)$ if $x\in \ch(A)_-$ and $\lambda\in \bbT$.
\end{proof}


\section{An additive Bishop's lemma}\label{sec5}
Let $R$ be a closed solid rhombus in the complex plane ${\mathbb C}$ whose four vertices are $0,\frac{\sqrt{3}}{3}\exp({\pm\frac{\pi}{6}i}), 1$. Let $\widetilde{R}$ be a  closed solid hexagon in ${\mathbb C}$ whose six vertices are $0,\frac{1}{3}\exp(\pm\frac{\pi}{3}i), \frac{\sqrt{3}}{3}\exp(\pm\frac{\pi}{6}i), 1$. 
The following lemma holds. 
\begin{lemma}\label{diamond}
The following two statements hold.
\begin{itemize}
\item[(i)] If $z\in \widetilde{R}$, then $|2z-1|\le 1$.
\item[(ii)] If $z,w\in R$, then $zw\in \widetilde{R}$.
\end{itemize}
\end{lemma}
\begin{proof}
It is trivial that (i) holds. 
We prove (ii). Since $R$ is the convex hull of the vertices $0,\frac{\sqrt{3}}{3}\exp({\pm\frac{\pi}{6}i})$ and  $1$, the set $\{zw: z,w\in R\}$ is the convex hull of the products of these vertices: $0, \frac13, \frac{\sqrt{3}}{3}\exp({\pm\frac{\pi}{6}i}), \frac{1}{3}\exp(\pm\frac{\pi}{3}i)$ and $1$, which coincides with $\widetilde{R}$.
\end{proof}
In the following,
\[
D=\{z\in {\mathbb C}:|z|<1\}, \quad \overline{D}=\{z\in {\mathbb C}:|z|\le 1\}.
\]
Let $x\in \ch(A)$. Put $P_x=\{f\in S(A): f(x)=1, f^{-1}(1)=|f|^{-1}(1)\}$. In fact, the set $P_x$ is the set of all peaking functions in $A$ which peaks at $x$. 
\begin{lemma}\label{tamasi}
Let $x\in \ch(A)$, $G$ an open neighborhood of $x$ in $X$ and $\delta>0$.
Then there exists $u\in A$ which satisfies the following.
\begin{itemize}
\item[(1)]
$1=u(x)=\|u\|$, $u^{-1}(1)=|u|^{-1}(1)$,
\item[(2)]
$|u|<\delta$ on $X\setminus G$,
\item[(3)]
$u(X)\subset R$.
\end{itemize}
\end{lemma}
\begin{proof}
Let $\pi_0:\overline{D}\to R$ be a homeomorphism such that $\pi_0|D$ is an analytic map from $D$ onto the interior of $R$. Such a homeomorphism exists by the well known theorem of Carath\'eodory (cf. \cite{carathe}). In fact $\pi_0$ is the inverse of the homeomorphic extension of the Riemann map from the interior of $R$ onto $D$. We may assume that $\pi_0(1)=1$ and $\pi_0(-1)=0$. Put $\Omega=\pi_0^{-1}(R\cap\{z\in {\mathbb C}:|z|<\delta\})$. Choose $0<r<1$ so that $\pi_{\delta}(\{z\in {\mathbb C}:|z|<1/2\})\subset \Omega$, where $\pi_{\delta}(z)=\frac{z-r}{1-rz}$. It is possible if $|1-r|$ is sufficiently small.  Put $\pi=\pi_0\circ\pi_{\delta}$. 
Choose a peaking function $f_x\in P_x$. We may assume that 
$|f_x|< 1/2$ on $X\setminus G$. Put $u=\pi\circ f_x$. As $\pi$ is approximated by analytic polynomials on $\overline{D}$, we have $u\in A$. It is automatic that (1), (2) and (3) hold.
\end{proof}
The following is an additive Bishop's lemma.
\begin{lemma}\label{additive}
Let $f\in S(A)$, $x\in \ch(A)$. Put $f(x)=\alpha$. For any $0<r<1$ there exists $u_r\in P_x$ such that $g_{r+}=\left(\frac{\alpha}{|\alpha|}-r\alpha\right)u_r+rf$ and $g_{r-}=\left(\frac{-\alpha}{|\alpha|}-r\alpha\right)u_r+rf$ are elements in $S(A)$ with  $g_{r+}(x)=\frac{\alpha}{|\alpha|}$, $g_{r-}(x)=\frac{-\alpha}{|\alpha|}$. Here $\frac{\alpha}{|\alpha|}$ reads $1$ if $\alpha=0$.
\end{lemma}
\begin{proof}
Let $0<r<1$ and $0<\varepsilon<1-r|\alpha|$. Put
\[
F_0=\{y\in X:|r\alpha -rf(y)|\ge \varepsilon/4\}.
\]
For a positive integer $n$ put
\[
F_n=\{y\in X:\varepsilon/2^{n+2}\le |r\alpha-rf(y)|\le \varepsilon/2^{n+1}\}.
\]
Then by Lemma \ref{tamasi} there exists $u_0\in P_x$ such that $u_0(X)\subset R$ and $|u_0|<\frac{1-r}{1+r|\alpha|}$ on $F_0$. 
For every positive integer $n$ there exists $u_n\in P_x$ such that $u_n(X)\subset R$ and $|u_n|<\frac{1}{2^{n+1}}$ on $F_n$. 
Put $u_r=u_0\sum_{n=1}^\infty\frac{u_n}{2^n}$. Note that the sum uniformly converges since $\|u_n\|=1$. Then $u_r$ is the desired function. We prove this. 

As $u_n(X)\subset R$ for all positive integer $n$ and $R$ is a convex set containing $0$, we have that  $\left(\sum\frac{u_n}{2^n}\right)(X)\subset R$. Since $u_0(X)\subset R$, we obtain that $u_r(X)\subset \widetilde{R}$ and $\|1-2u_r\|\le 1$ by Lemma \ref{diamond}. 
We prove $|g_{r+}(y)|\le 1$ and $|g_{r-}(y)|\le 1$ for any $y\in X$ in three cases: 1) $y\in F_0$; 2) $y\in F_n$ for some positive integer $n$; 3) $y\in X\setminus (U_{n=1}^\infty F_n)\cup F_0$. 

1) Suppose that $y\in F_0$. We have $|u_r(y)|\le|u_0(y)|<\frac{1-r}{1+r|\alpha|}$. 
As $\left|\frac{\alpha}{|\alpha|}-r\alpha\right|=1-r|\alpha|$ and $\left|\frac{-\alpha}{|\alpha|}-r\alpha\right|=1+r|\alpha|$, we see that
\[
|g_{r+}(y)|\le (1-r|\alpha|)\frac{1-r}{1+r|\alpha|}+r|f(y)|<1,
\]
\[
|g_{r-}(y)|\le (1+r|\alpha|)\frac{1-r}{1+r|\alpha|}+r|f(y)|\le 1
\]
since $\|f\|=1$.

2) Suppose that $y\in F_n$ for some positive integer $n$. Then 
\begin{multline*}
|u_r(y)|\le \sum_{m=1}^\infty\frac{|u_m(y)|}{2^m}=\sum_{m\ne n}\frac{|u_m(y)|}{2^m}+\frac{|u_n(y)|}{2^n} \\
\le
\sum_{m\ne n}\frac{1}{2^m}+\frac{1}{2^n2^{n+1}}=1-\frac{1}{2^n}+\frac{1}{2^n2^{n+1}}<1.
\end{multline*}
Then
\begin{align*}
|g_{r+}(y)|
&\le \left|\left(\frac{\alpha}{|\alpha|}-r\alpha\right)u_r(y)+r\alpha\right|+|rf(y)-r\alpha| \\
&\le(1-r|\alpha|)|u_r(y)|+r|\alpha|+\frac{\varepsilon}{2^{n+1}} \\
&\le(1-r|\alpha|)\left(1-\frac{1}{2^n}+\frac{1}{2^n2^{n+1}}\right)+r|\alpha|+\frac{\varepsilon}{2^{n+1}} 
\end{align*}
as $0<\varepsilon<1-r|\alpha|$
\begin{align*}
\hspace{7mm}
&\le (1-r|\alpha|)|\left(1-\frac{1}{2^n}+\frac{1}{2^n2^{n+1}}+\frac{1}{2^{n+1}}\right)+r|\alpha| \\
&<(1-r|\alpha|)+r|\alpha|=1.
\end{align*}
We also have
\begin{align*}
|g_{r-}(y)|&\le
\left|\left(\frac{-\alpha}{|\alpha|}+r\alpha\right)u_r(y)-2r\alpha u_r(y)+r\alpha\right|+|rf(y)-r\alpha| \\
&\le \left|\frac{-\alpha}{|\alpha|}+r\alpha\right||u_r(y)|+r|\alpha||1-2u_r(y)|+\frac{\varepsilon}{2^{n+1}} \\
\end{align*}
as $\|1-2u_r\|\le 1$ and $0<\varepsilon<1-r|\alpha|$ we have
\begin{align*}
\hspace{7mm}&
\le(1-r|\alpha|)\left(1-\frac{1}{2^n}+\frac{1}{2^n2^{n+1}}+\frac{1}{2^{n+1}}\right)+r|\alpha|<1.
\end{align*}

3) Suppose that 
$y\in X\setminus (U_{n=1}^\infty F_n)\cup F_0$. 
In this case $rf(y)=r\alpha$. 
We have
\[
|g_{r+}(y)|\le(1-r|\alpha|)|u_r(y)|+r|\alpha|\le 1
\]
since $\|u_r\|=1$.  We also have
\begin{align*}
|g_{r-}(y)|&=\left|\left(\frac{-\alpha}{|\alpha|}+r\alpha\right)u_r(y)-2r\alpha u_r(y)+r\alpha\right| \\
&\le(1-r|\alpha|)|u_r(y)|+r|\alpha||2u_r(y)-1|\le 1.
\end{align*}
As $g_{r+}(x)=\frac{\alpha}{|\alpha|}$ and $g_{r-}(x)=\frac{-\alpha}{|\alpha|}$ since $u_r(x)=1$,
we observe by 1), 2) and 3) that $\|g_{r+}\|=\|g_{r-}\|=1$, hence $g_{r+},g_{r-}\in S(A)$.

\end{proof}

\section{A form of $Tf$ on the Choquet boundary}\label{sec6}

\begin{definition}
Let $x\in \ch (A)$ and $\alpha\in \overline{D}$.
A subset $M_{x,\alpha}$ of $S(A)$ is defined as
\[
M_{x,\alpha}=\{f\in S(A): d(f, F_{x,\frac{\alpha}{|\alpha|}})=1-|\alpha|,\,\,d(f,F_{x,\frac{-\alpha}{|\alpha|}})=1+|\alpha|\},
\]
where $d(f, F)=\inf\{\|f-g\| : g\in F\}$ and $\frac{\alpha}{|\alpha|}$ reads $1$ if $\alpha=0$.
\end{definition}
The set $M_{x,\alpha}$ corresponds to  $E_{\lambda}$ appeared in \cite{moriozawa}.
\begin{lemma}\label{Tmxa}
For every pair $x\in \ch(A)$ and $\lambda\in \bbT$,  the equality 
\[
d(f, F_{x,\lambda})=d(T(f),F_{\phi(x),\tau(x,\lambda)}).
\]
holds for every $f\in S(A)$.
\end{lemma}
\begin{proof}
As $T$ is an isometry we have $d(f,g)=d(T(f),T(g))$ for every $g\in F_{x,\lambda}$.
Thus $d(f, F_{x,\lambda})=d(T(f),T(F_{x,\lambda}))$. By Lemma \ref{mapsbtmcs} we have
$d(f, F_{x,\lambda})=d(T(f),F_{\phi(x),\tau(x,\lambda)})$.
\end{proof}
\begin{lemma}\label{mxa}
For every pair $x\in \ch(A)$ and $\alpha\in \overline{D}$, we have
$\mxa=\{f\in S(A):f(x)=\alpha\}$.
\end{lemma}
\begin{proof}
Suppose that $f\in \mxa$. For every $h\in \fxa$ we have
\[
\left|f(x)-\frac{\alpha}{|\alpha|}\right|=|f(x)-h(x)|\le \|f-h\|.
\]
Hence
\begin{equation}\label{mxa1}
\left|f(x)-\frac{\alpha}{|\alpha|}\right|\le d(f,\fxa)=1-|\alpha|.
\end{equation}
For every $g\in \fxma$ we have
\[
\left|f(x)-\frac{-\alpha}{|\alpha|}\right|=|f(x)-g(x)|\le \|f-g\|.
\]
Hence
\begin{equation}\label{mxa2}
\left|f(x)-\frac{-\alpha}{|\alpha|}\right|\le d(f,\fxma)=1+|\alpha|.
\end{equation}
By \eqref{mxa1} and \eqref{mxa2} we see that $f(x)=\alpha$.

Suppose that $f\in S(A)$ and $f(x)=\alpha$. 
For any $0<r<1$, 
the functions $g_{r+}$ and $g_{r-}$, defined in Lemma \ref{additive}, 
belong to $\fxa$ and $\fxma$, respectively.
We have
\[
\|g_{r+}-f\|\le (1-r|\alpha|)\|u_r\|+\|rf-f\|=(1-r|\alpha|)+1-r
\]
and
\[
\|g_{r-}-f\|\le (1+r|\alpha|)\|u_r\|+\|rf-f\|=(1+r|\alpha|)+1-r.
\]
As $r$ is an arbitrary number of the open interval $(0,1)$, we have $d(f, \fxa)\le 1-|\alpha|$ and $d(f,\fxma)\le 1+|\alpha|$. 
On the other hand, 
\[
1-|\alpha|=\left|\alpha- \frac{\alpha}{|\alpha|}\right|=|f(x)-h(x)|\le \|f-h\|
\]
for any $h\in \fxa$ and 
\[
1+|\alpha|=\left|\alpha-\frac{-\alpha}{|\alpha|}\right|=|f(x)-h'(x)|\le \|f-h'\|
\]
for any $h'\in \fxma$. It follows that  $d(f,\fxa)=1-|\alpha|$ and $d(f,\fxma)=1+|\alpha|$. 
Hence we get $f\in M_{x,\alpha}$.

\end{proof}
\begin{lemma}\label{Tf}
For every $f\in S(A)$ we have
\begin{equation*}
T(f)(\phi(x))=\tau(x,1)\times
\begin{cases}
f(x),\qquad x\in \ch(A)_+ \\
\overline{f(x)}, \qquad x\in \ch(A)_-.
\end{cases}
\end{equation*}
\end{lemma}
\begin{proof}
Let $x\in \ch(A)$ and $\alpha\in \overline{D}$.
As $T(S(A))=S(B)$ we have by Lemma \ref{Tmxa} that
\begin{align*}
&T(\mxa)=\{g\in S(B):d(g,F_{\phi(x),\tau(x,\frac{\alpha}{|\alpha|})})=1-|\alpha|,\,\,d(g,F_{\phi(x),\tau(x,\frac{-\alpha}{|\alpha|})})=1+|\alpha|\}, \\
&\text{then by Lemma \ref{txl}}\\
&=
\begin{cases}
\{g:d(g,\Tfxa)=1-|\alpha|,\,\,d(g,\Tfxma)=1+|\alpha|\},\quad x\in \ch(A)_+\\
\{g:d(g,\Tfxba)=1-|\alpha|,\,\,d(g,\Tfxmba)=1+|\alpha|\},\quad x\in \ch(A)_-
\end{cases}\\
&=
\begin{cases}
M_{\phi(x),\alpha\tau(x,1)},\qquad x\in \ch(A)_+ \\
M_{\phi(x),\overline{\alpha}\tau(x,1)},\qquad x\in \ch(A)_-,
\end{cases}\\
&\text{applying Lemma \ref{mxa} for $M_{\phi(x),\beta}\subset S(B)$ for $\beta=\alpha\tau(x,1),\,\,\overline{\alpha}\tau(x,1)$ we have}\\
&=
\begin{cases}
\{g\in S(B):g(\phi(x))=\alpha\tau(x,1)\},\qquad x\in \ch(A)_+ \\
\{g\in S(B):g(\phi(x))=\overline{\alpha}\tau(x,1)\},\qquad x\in \ch(A)_-.
\end{cases}
\end{align*}
Let $f\in S(A)$. Put $\alpha=f(x)$. Then $f\in M_{x,\alpha}$ by Lemma \ref{mxa}. By the above we have $Tf(\phi(x))=\alpha\tau(x,1)$ if $x\in \ch(A)_+$ and 
$Tf(\phi(x))=\overline{\alpha}\tau(x,1)$ if $x\in \ch(A)_-$. As $f(x)=\alpha$ we have the conclusion.
\end{proof}
\section{Proof of Theorem \ref{main}}\label{sec7}
\begin{proof}[Proof of Theorem \ref{main}]
By Lemma \ref{mapsbtmcs} the map $\phi:\ch(A)\to \ch(B)$ is a bijection. Denote the inverse of $\phi$ by $\psi:\ch(B)\to \ch(A)$.
Then by Lemma \ref{Tf} we have
\begin{equation}\label{primitiveform}
T(f)(y)=\tau(\psi(y),1)\times
\begin{cases}
f\circ\psi(y),\qquad y\in \phi(\ch(A)_+) \\
\overline{f\circ\psi(y)}, \qquad y\in \phi(\ch(A)_-)
\end{cases}
\end{equation}
for every $f\in S(A)$. Note that since $\phi:\ch(A)\to \ch(B)$ is a bijection we have $\phi(\ch(A)_+)\cap\phi(\ch(A)_-)=\emptyset$ and 
 $\phi(\ch(A)_+)\cup\phi(\ch(A)_-)=\ch(B)$.
 
 We prove that $T(1)$ is invertible and $|T(1)|=1$ on $M_B$. Since $T$ is a surjection, there exists $f_1\in S(A)$ with $T(f_1)=1$.
 By \eqref{primitiveform} we have
 \begin{equation}\label{7}
 T(1)(y)=\tau(\psi(y),1),\qquad y\in \ch(B),
 \end{equation}
 and
 \begin{equation}\label{8}
1= T(f_1)(y)=\tau(\psi(y),1)\times
 \begin{cases}
f_1\circ\psi(y),\qquad y\in \phi(\ch(A)_+) \\
\overline{f_1\circ\psi(y)}, \qquad y\in \phi(\ch(A)_-),
\end{cases}
\end{equation}
and 
\begin{equation}\label{9}
T(f_1^2)(y)=\tau(\psi(y),1)\times
 \begin{cases}
(f_1)^2\circ\psi(y),\qquad y\in \phi(\ch(A)_+) \\
\overline{(f_1)^2\circ\psi(y)}, \qquad y\in \phi(\ch(A)_-),
\end{cases}
\end{equation}
Combining \eqref{7}, \eqref{8} and \eqref{9} we get
\[
\text{$1=(T(f_1))^2=T(1)T(f_1^2)$ on $\ch(B)$}.
\]
 Since $\ch(B)$ is a uniqueness set for $B$, we obtain
 \[
 \text{$1=T(1)T(f_1^2)$ on $M_B$}.
 \]
 Thus $T(1)$ is invertible in $B$. On the other hand, $|T(1)|=1$ on $\ch(B)$ induces that $|(T(1))^{-1}|=1$ on $\ch(B)$. Thus 
 $|T(1)|\le 1$ and $|(T(1))^{-1}|\le 1$ on $M_B$. It follows that $|T(1)|=1$ on $M_B$. 
 
 Define $T_1:A\to B$ by 
 \begin{equation*}
 T_1(f) =\overline{T(1)}\times
 \begin{cases}
 0, \qquad f=0 \\
 \|f\|T\left(\frac{f}{\|f\|}\right),\quad f\ne 0.
 \end{cases}
 \end{equation*}
 By a simple calculation $T_1$ is a bijection since $T$ is. It is easy to see that $T_1$ is an extension of $\overline{T(1)}T$, hence $T(1)T_1$ is an extension of $T$.
 By \eqref{primitiveform} and \eqref{7} we infer that 
 \begin{equation*}
 T_1(f)=
 \begin{cases}
 \text{$f\circ\psi$ on $\phi(\ch(A)_+)$} \\
 \text{$\overline{f\circ\psi}$ on $\phi(\ch(A)_-)$}.
 \end{cases}
 \end{equation*}
 As $\ch(B)$ is a uniqueness set for $B$, we infer that $T_1$ is a real-linear algebra-isomorphism from $A$ onto $B$. 
 Then by \cite[Theorem 2.1]{hm} there exist a homeomorphism $\Psi:M_B\to M_A$, possibly empty disjoint closed and open subsets 
 $M_{B+}$ and $M_{B-}$ of $M_B$ with $M_{B+}\cup M_{B-}=M_B$ such that
 \begin{equation*}
 T_1(f)=
 \begin{cases}
 \text{$f\circ\Psi$ on $M_{B+}$} \\
 \text{$\overline{f\circ\Psi}$ on $M_{B-}$}
 \end{cases}
 \end{equation*}
 for every $f\in A$. It follows that 
 \begin{equation*}
T(1) T_1(f)=T(1)\times
 \begin{cases}
 \text{$f\circ\Psi$ on $M_{B+}$} \\
 \text{$\overline{f\circ\Psi}$ on $M_{B-}$}
 \end{cases}
 \end{equation*}
 for every $f\in A$ is a surjective real-linear isometry from $A$ onto $B$ since $|T(1)|=1$ on $M_B$. As is already pointed out that 
 $T(1)T_1$ is an extension of $T$ we conclude the proof.

\end{proof}
\section{Remarks}\label{sec8}
We close the paper with a few remarks. As the first one we conjecture that a uniform algebra satisfies the Mazur-Ulam property. A surjective isometry between the unit spheres of uniform algebras is represented by a ``weighted composition operator'', because a surjective real-linear isometry between uniform algebras is represented as in this form. In fact, the point of the proof of Theorem \ref{main} is to show the form of the given isometry between the unit spheres.  In general, a surjective real-linear isometry between a uniform algebra and a Banach space of continuous functions is not expected to have the form of a ``weighted composition operator''. We do not know how to prove the conjecture.

The second remark concerns Tingley's problem on a Banach space of analytic functions. As we have already pointed out that Theorem \ref{main} provides the first positive solution to Tingley's problem on a Banach space of analytic functions. It is interesting to study Tingley's problem on several Banach spaces of analytic functions. The form of a surjective complex-linear isometry on the Hardy space $H^1(D)$ on the open unit disk is derived by a theorem of deLeeuw, Rudin and Wermer \cite{drw} (cf. \cite{nagasawa}) on the isometries between uniform algebras. The point of  the proof is to show that an isometry $U:H^1(D)\to H^1(D)$ essentially preserves the range of functions. It reminds us that if we can prove that the range of the function is essentially preserved by the isometry $T:S(H^1(D))\to S(H^1(D))$, we have a chance to solve Tingley's problem on $H^1(D)$.

As a final remark we encourage research on Tingley's problem on several Banach algebras of continuous functions. Comparing with the theorem of Wang \cite{wang1996a} on the Banach algebra of $C^{(n)}(X)$, it is interesting to study a surjective isometry on the unit sphere of a Banach space or algebra of Lipschitz functions, it has already been pointed out by Cueto-Avellaneda \cite[Problem 4.0.8]{cueto}.

\subsection*{Acknowledgments}
The authors record their sincerest appreciation to the three referees for their valuable comments and advice which have improved the presentation of this paper substantially. 
The first author was supported by JSPS KAKENHI Grant Numbers JP19K03536.  The second author was supported by JSPS KAKENHI Grant Numbers JP21K13804.


\begin{thebibliography}{99}


\bibitem{bcfp}
J.~Becerra-Guerrero, M.~Cueto-Avellaneda, F.~J.~Fern\'andez-Polo and A.~M.~Peralta,
\emph{
On the extensin of isometries between the unit spheres of a JBW*-triple and a Banach space},
J. Inst. Math. Jussieu Published online 15 April 2019
doi:10.1017/S1474748019000173

\bibitem{br}
A.~Browder,
\emph{
Introduction to function algebras},
W. A. Benjamin, Inc., New York-Amsterdam 1969 Xii+273 pp

\bibitem{chengdong}
L.~Cheng and Y.~Dong,
\emph{
On a generalized Mazur-Ulam question: extension of isometries between unit spheres of Banach spaces},
J. Math. Anal. Appl. {\bf 377} (2011), 464--470
doi:10.1016/j.jmaa.2020.11.025

\bibitem{cueto}
M.~Cueto-Avellaneda,
\emph{
Extension of isometreis and the Mazur-Ulam property},
PhD thesis, Universidad de Almer\'ia,  2020



\bibitem{cp2019}
M.~Cueto-Avellaneda and A.~M.~Peralta,
\emph{
On the Mazur-Ulam property for the space of Hilbert-space-valued continuous functions},
J. Math. Anal. Appl. {\bf479} (2019), 875--902
doi:10.1016/j.jmaa.2019.06.056

\bibitem{drw}
K.~deLeeuw, W.~Rudin and J.~Wermer,
\emph{
The isometries of some function spaces},
doi:10.1090/S0002-9939-1960-0121646-9





\bibitem{ding2003B}
G.~G.~Ding,
\emph{
On the extension of isometries between unit spheres of $E$ and $C(\Omega)$},
Acta Math. Sin. (Engl. Ser.)  {\bf19} (2003), 793--800
doi:10.1007/s10114-003-0240-z






\bibitem{ding2009}
G.~G.~Ding,
\emph{
On isometric extension problem between two unit spheres},
Sci. Chin. Ser. A {\bf 52} (2009), 2069--2083
doi:10.1007/s11425-009-0156-x






\bibitem{fgpv2017}
F.~J.~Fern\'andez-Polo, J.~J.~Garc\'es, A.~M.~Peralta and I.~Villanueva,
\emph{
Tingley's problem for spaces of trace class operators},
Linear Algebra Appl. {\bf529} (2017), 294--323
doi:10.1016/j.laa.2017.04.024

\bibitem{fjp2020}
F.~J.~Fern\'andez-Polo, E.~Jord\'a and A.~M.~Peralta,
\emph{
Tingley's problem for $p$-Schatten von Neumann classes},
J. Spectr. Theory {\bf10} (2020), 809--841
doi:10.4171/JST/313


\bibitem{fp2018a}
F.~J.~Fern\'andez-Polo and A.~M.~Peralta,
\emph{
On the extension of isometries between the unit spheres of a $C^*$-algebra and $B(H)$},
Trans. Amer. Math. Soc. Ser. B {\bf 5} (2018), 63--80
doi:10.1090/btran/21

\bibitem{fp2018b}
F.~J.~Fern\'andez-Polo and A.~M.~Peralta,
\emph{
On the extension of isometries between the unit spheres of von Neumann algebras},
J. Math. Anal. Appl. {\bf466} (2018), 127--143
doi:10.1016/j.jmaa.2018.05.062

\bibitem{fp2018c}
F.~J.~Fern\'andez-Polo and A.~M.~Peralta,
\emph{
Low rank compact operators and Tingley's problem},
Adv. Math. {\bf338} (2018), 1--40
doi:10.1016/j.aim.2018.08.018

\bibitem{fp2017}
F.~J.~Fern\'andez-Polo and A.~M.~Peralta,
\emph{
Tingley's problem through the facial structure of an atomic JBW*-triple},
J. Math. Anal. Appl. {\bf455} (2017), 750--760
doi:10.1016/j.jmaa.2017.06.002

\bibitem{fj1}
R.~J.~Fleming and J.~E.~Jamison,
\emph{
Isometries on Banach spaces: function spaces},
 Chapman \& Hall/CRC Monographs and Surveys in Pure and Applied Mathematics, 129. Chapman \& Hall/CRC, Boca Raton, FL, 2003. x+197 pp. ISBN: 1-58488-040-6
 
 
 \bibitem{hm}
 O.~Hatori and T.~Miura,
 \emph{
 Real linear isometries between function algebras. II},
 Cent. Eur. J. Math. {\bf 11} (2013), 1838--1842
 doi:10.2478/s11533-013-0282-0
 
 
 
 \bibitem{jmpr2019}
 A.~Jim\'enez-Vargas, A.~Morales-Campoy, A.~M.~Peralta and M.~I.~Ram\'irez,
 \emph{
 The Mazur-Ulam property for the space of complex null sequences},
 Linear Multilinear Algebra {\bf 67} (2019), 799--816
 doi:10.1080/03081087.2018.1433625
 

 
 \bibitem{kp}
 O.~F.~K.~Kalenda and A.~M.~Peralta,
 \emph{
 Extension of isometries from the unit sphere of a rank-2 Cartan factor},
 Anal. Math. Phys. {\bf 11}, Article number:15 (2021)
doi:10.1007/s13324-020-00448-2
 
  \bibitem{lnw2020}
 A.~T.~M.~Lau, C.~K.~Ng and N.~C.~Wong,
 \emph{
 Normal states are determined by their facial distances},
 Bull. Lond. Math. Soc. {\bf52} (2020), 505--514
 doi:10.1112/blms.12344
 
 \bibitem{lnw2018}
 C.~W.~Leung, C.~K.~Ng and N.~C.~Wong,
 \emph{
 Metric preserving bijections between positive spherical shells of non-commutative $L^p$-spaces},
 J. Operator Theory {\bf 80} (2018), 429--452
 
 \bibitem{lnw2021}
 C.~W.~Leung, C.~K.~Ng and N.~C.~Wong,
 \emph{
 On a variant of Tingley's problem for some function spaces},
 J. Math. Anal. Appl. {\bf496} (2021) 124800
 doi:10.1016/j.jmaa.2020.124800
 
 
 
 
 
 
 
 
\bibitem{mp}
M.~Mart\'in and R.~Pay\'a,
\emph{
On CL-spaces and almost CL-spaces}, 
Ark. Mat. {\bf42} (2004), 107--118
doi:10.1007/BF02432912

\bibitem{mori}
M.~Mori,
\emph{
Tingley's problem through the facial structure of operator algebras},
J. Math. Anal. Appl. {\bf 466} (2018), 1281--1298
doi:10.1016/j.jmaa.2018.06.050

\bibitem{moriozawa}
M.~Mori and N.~Ozawa,
\emph{
Mankiewicz's theorem and the Mazur-Ulam property for $C^*$-algebras},
Studia Math. {\bf 250} (2020), 265--281
doi:10.4064/sm180727-14-11

\bibitem{nagasawa}
M.~Nagasawa,
\emph{
Isomorphisms between commutative Banach algebras with an application to rings of analytic functions},
K\= odai Math. Sem. Rep.,
{\bf 11} (1959), 182--188
doi:10.2996/kmj/1138844205

\bibitem{nagy2018}
G.~Nagy,
\emph{
Isometries of spaces of normalized positive operators under the operator norm},
Publ. Math. Debrecen {\bf92} (2018), 243--254
doi:10.5486/PMD.2018.7967

\bibitem{peralta2018}
A.~M.~Peralta,
\emph{
A survey on Tingley's problem for operator algebras}, 
Acta Sci. Math. (Szeged) {\bf84} (2018), 81--123
doi:10.14232/actasm-018-255-0

\bibitem{peralta2019a}
A.~M.~Peralta,
\emph{
Extending surjective isometries defined on the unit sphere of $\ell_\infty(\Gamma)$}
Rev. Mat. Complut. {\bf32} (2019), 99--114
doi:10.1007/s13163-018-0269-2

\bibitem{peraltaJB*}
A.~M.~Peralta,
\emph{
On the extension of surjective isometries whose domain is the unit sphere of a space of compact operators},
preprint, 2020, 
arXiv:2005.11987v1

\bibitem{peralta2019b}
A.~M.~Peralta,
\emph{
On the unit sphere of positive operators},
Banach J. Math. Anal. {\bf13} (2019), 91--112
doi:10.1215/17358787-2018-0017

\bibitem{pt2019}
A.~M.~Peralta and R.~Tanaka,
\emph{
A solution to Tingley's problem for isometries between the unit spheres of compact $C^*$-algebras and $JB^*$-triples},
Sci. China Math. {\bf62} (2019), 553--568
doi:10.1007/s11425-017-9188-6

\bibitem{carathe}
Ch.~Pomerenke,
\emph{
Boundary behaviour of conformal maps}, 
Grundlehrender Mathematischen Wissenschaften, 299 
Springer-Verlag, Berlin, 1992, ix+300 pp
doi:10.1007/978-3-662-02770-7



\bibitem{thl2013}
D.~Tan, X.~Huang and R.~Liu,
\emph{
Generalized-lush spaces and the Mazur-Ulam property},
Studia Math. {\bf219} (2013), 139--153
doi:10.4064/sm219-2-4

\bibitem{tl2013}
D.~Tan and R.~Liu,
\emph{
A note on the Mazur-Ulam property of almost-CL-spaces},
J. Math. Anal. Appl. {405} (2013), 336--341
doi:10.1016/j.jmaa.2013.03.024

\bibitem{tanaka2014a}
R.~Tanaka,
\emph{
Tingley's problem on symmetric absolute normalized norms on ${\mathbb R}^2$},
Acta Math. Sin. (Engl. Ser.) {\bf 30} (2014), 1324--1340
doi:10.1007/s10114-014-3491-y

\bibitem{tanaka2014b}
R.~Tanaka,
\emph{
A further property of spherical isometries},
Bull. Aust. Math. Soc. {\bf 90} (2014), 304--310
doi:10.1017/S0004972714000185

\bibitem{tanaka}
R.~Tanaka,
\emph{
The solution of Tingley's problem for the operator norm unit sphere of complex $n\times n$ matrices},
Linear Algebra Appl. {\bf 494} (2016), 274--285
doi:10.1016/j.laa.2016.01.020

\bibitem{tanaka2017}
R.~Tanaka,
\emph{
Spherical isometries of finite dimensional $C^*$-algebras},
J. Math. Anal. Appl. {\bf445} (2017), 337--341
doi:10.1016/j.jmaa.2016.07.073

\bibitem{tanaka2017b}
R.~Tanaka,
\emph{
Tingley's problem on finite von Neumann algebras},
J. Math. Anal. Appl. {\bf451} (2017), 319--326
doi:10.1016/j.jmaa2017.02.013

\bibitem{tingley}
D.~Tingley,
\emph{
Isometries of the unit sphere},
Geom. Dedicata {\bf 22} (1987), 371--378
doi:10.1007/BF00147924



\bibitem{wang}
R.~S.~Wang,
\emph{
Isometries between the unit spheres of $C_0(\Omega)$ type spaces},
Acta Math. Sci. (English Ed.) {\bf 14} (1994), 82--89
doi:10.1016/S0252-9602(18)30093-6


\bibitem{wang1996a}
Risheng Wang,
\emph{
Isometries of $C_0^{(n)}(X)$},
Hokkaido Math. J. {\bf 25} (1996), 465--519
doi:10.14492/hokmj/1351516747





\bibitem{wh2019}
Ruidong Wang and X.~Huang,
\emph{
The Mazur-Ulam property for two dimensional somewhere-flat spaces},
Linear Algebra Appl. {\bf562} (2019), 55--62
doi:10.1016/j.laa.2018.09.024

\bibitem{yang2006}
X.~Yang,
\emph{
On extension of isometries between unit spheres of $L_p(\mu)$ and $L_p(\mu, H)$ $(1<p\ne2$, $H$ is a Hilbert space$)$},
J. Math. Anal. Appl. {\bf 323} (2006), 985--992
doi:10.1016/j.jmaa.2005.11.013

\bibitem{yangzhao}
X.~Yang and X.~Zhao,
\emph{
On the extension problems of isometric and nonexpansive mappings},
In:Mathematics without boundaries. Edited by Themistocles M. Rassias and Panos M. Pardalos, 725-- Springer, New York, 2014
doi:10.1007/978-1-4939-1106-6$\underline{\,\,\,}$24






\end{thebibliography}
\end{document}